\documentclass[11pt,a4paper]{article}
\usepackage[ansinew]{inputenc}
\usepackage{enumerate}
\usepackage[spanish,english]{babel}
\usepackage{amsthm,amssymb,amsfonts,amsmath}
\usepackage{tkz-graph}
\usepackage{color}

\newtheorem{theorem}{Theorem}[section]
\newtheorem{lemma}[theorem]{Lemma}

\newtheorem{proposition}[theorem]{Proposition}
\newtheorem{corollary}[theorem]{Corollary}

\title{The resolving number of a graph}
\author{D. Garijo \thanks{Department of
Applied Mathematics I, University of Seville, Spain. Email
addresses: {\tt \{dgarijo,gonzalezh,almar\}@us.es}. The first and second authors were partially supported by projects 2009/FQM-164 and 2010/FQM-164. The third author was partially supported by the ESF EUROCORES programme EuroGIGA ComPoSe IP04 MICINN Project EUI-EURC-2011-4306, and projects 2009/FQM-164 and 2010/FQM-164.} \and A.
González$^*$ \and A. Márquez$^*$}

\date{}

\begin{document}
\maketitle

\begin{abstract}

We study a graph parameter related to resolving sets and  metric dimension, namely the resolving number, introduced by Chartrand, Poisson and Zhang. First, we establish an important difference between the two parameters: while computing the metric dimension of an arbitrary graph is known to be NP-hard, we show that the resolving number can be computed in polynomial time. We then relate the resolving number to classical graph parameters: diameter, girth, clique number, order and maximum degree. With these relations in hand, we characterize the graphs with resolving number 3 extending other studies that provide characterizations for smaller resolving number.
\end{abstract}

\section{Introduction}\label{sec:intro}

Let $G=(V(G),E(G))$ be a finite, simple, undirected and connected graph of order $n=|V(G)|$.
The \emph{ distance} $d(u,v)$ between two vertices $u, v \in V(G)$ is the length of a
shortest $u$-$v$ path in $G$. A vertex $u\in V(G)$ \emph{ resolves} a pair $\{x,y\}\subset V(G)$ if $d(u,x)\neq d(u,y)$. A
set of vertices $S\subseteq V(G)$ is a \emph{ resolving set} of $G$ if every pair of vertices of $G$ is resolved by some
vertex in $S$. The \emph{ metric dimension} of $G$, denoted by ${\rm dim} (G)$, is the minimum cardinality of a resolving set of $G$.
The \emph{ resolving number}, written as ${\rm res} (G)$,
is the minimum $k$ such that every $k$-subset of $V(G)$ is a resolving set of $G$. Obviously, every set $S\subseteq V(G)$ with $|S|\geq {\rm res} (G)$ is a resolving set of $G$.

Resolving sets and metric dimension were first introduced by Slater~\cite{slater}, and independently by Harary and Melter~\cite{melter}. Much latter, Chartrand et al.~\cite{upper} defined the resolving number. There exists by now
an extensive literature on resolving sets and  the resolving parameters related to them, including applications to several areas such as coin weighing problems, pharmaceutical chemistry, robot navigation, network discovery and verification, and problems of pattern recognition and image processing. See for instance \cite{beerliova,cartesian,chartrand,upper,MDUDRN,extremal,landmarks,tomescu}.

Although the metric dimension and the resolving number are closely related by their definitions, there are significant differences between them. For instance, one can easily find infinite families of graphs with the same metric dimension, whereas in \cite{MDUDRN} we showed that the set of graphs with fixed resolving number $a\geq 4$ is finite.
This paper establishes another important difference between these parameters: while computing the metric dimension of an arbitrary graph is  NP-hard (see \cite{landmarks}), we prove that the resolving number can be computed in polynomial time.

We also deal with two types of problems typically studied for resolving parameters. First, we relate the resolving number to classical graph parameters: diameter, girth, clique number, order and maximum degree. Concretely, we provide bounds (most of which are tight) on the girth, the order, and the maximum degree of a graph in terms of its resolving number. We also determine the graphs $G$ with clique number equal to ${\rm res}(G)$, and restrict ourselves to trees to give bounds on the diameter, the order and the maximum degree, characterizing also the extremal cases.
Our study follows the same spirit as several papers that treat  analogous problems for the metric dimension. See \cite{chartrand,extremal,landmarks,yushmanov} for relations with the order and the diameter, and \cite{gimbel,tomescu2} for relations with the chromatic number and the partition dimension. Further, as it will be specified later, most of our results either improve relations obtained in other papers or continue with the studies developed in them.

Our second main problem is to characterize the graphs with given resolving number.
 As a consequence of some results in \cite{upper} and \cite{Jannesarion}, one can easily prove that the only graphs $G$ with ${\rm res}(G)\leq 2$ are paths and odd cycles. As a next step, we determine all graphs with resolving number 3 using as main tools
our relations between the resolving number and the graph parameters mentioned above.
 Similar results have been obtained for the metric dimension, and also combining metric dimension with resolving number: Chartrand et al.~\cite{chartrand} characterized the graphs with metric dimension equal to 1, $n-2$ and $n-1$, and the graphs $G$ such that ${\rm dim} (G)={\rm res}(G)=k$ are obtained by Garijo et al.~\cite{MDUDRN} (see also \cite{Jannesarichar}).

The paper is organized as follows. Section \ref{sec:section2} provides some notation, definitions and a series of technical results.  One of these results leads to Corollary \ref{coro} which states that ${\rm res}(G)$ can be computed in polynomial time. In Section \ref{sec:section3}, we present the above described relationships between the resolving number and classical graph parameters. As it was said before, besides their independent interest, these relations are the main tool to characterize the graphs with resolving number 3, which is done in Section \ref{sec:section4}. We conclude the paper in Section \ref{sec:section5} with some remarks and open problems.

\section{Preliminaries. Computational complexity of ${\rm res}(G)$}\label{sec:section2}

We begin by introducing some more notation and definitions.
 For $u\in V(G)$, we shall denote by $N(u)$ and $N[u]$ the open and closed neighbourhoods of $u$, respectively.
As usual, $\delta (u)$ is the degree of $u$, $\Delta (G)$ is the maximum degree of $G$, and $\langle A \rangle$ is the induced subgraph by a subset $A\subseteq V(G)$.

Let $P_n$ and $ C_n$ denote,  respectively, the path and the cycle on $n$ vertices. When no confusion can arise, we shall use $P$ (respectively, $C$ or $K$) to denote a path (respectively, cycle or clique) and also its vertex set.

The \emph{ girth} $g(G)$ is the minimum length of a cycle in $G$. Note that the girth of a tree is defined to be infinity.
The \emph{ clique number} $\omega (G)$ is the maximum size of a clique in $G$, and
the \emph{ diameter of $G$}, written as $d(G)$, is the maximum distance between any two vertices of $G$.
The \emph{ distance} between a vertex $u\in V(G)$ and a subset $A\subseteq V(G)$ is
 $ d(u,A)= \min_{v\in A}{d(u,v)}$, and the \emph{ diameter of
$A$} is $ d(A)= \max_{u,v\in A}{d(u,v)}$. The set of pairs of elements of $A$ is denoted by
$\mathcal{P}_2(A)$.

Let $R(x,y)$ be the set of vertices of $G$ that resolve the pair $\{x,y\}\subset V(G)$, and let $\overline{R}(x,y)=V(G)\setminus R(x,y)$. Note that for every pair $\{x,y\}$, the set $\overline{R}(x,y)$ is not a resolving set of $G$.

The following proposition is the key tool to obtain the previously mentioned difference on the computational complexity of
the resolving number and the metric dimension: ${\rm res}(G)$ is polynomial-time computable while computing ${\rm dim}(G)$ is NP-hard (see \cite{landmarks}).

\begin{proposition}\label{expr}
The resolving number of a graph $G$ is given by
$${\rm res} (G)= \max_{x,y\in V(G)}|\overline{R}(x,y)| +1.$$
\end{proposition}

\begin{proof} Let us denote $m=\max_{x,y\in V(G)} |\overline{R}(x,y)|$. Clearly, every set $S\subseteq V(G)$ with $|S|>m$ is a resolving set of $G$ (otherwise there would be a pair $\{x,y\}$ such that
$S\subseteq \overline{R}(x,y)$ and so  $m<|S|\leq |\overline{R}(x,y)|\leq m$). Hence, ${\rm res}(G)\leq m+1$. On the other hand, there is at least one pair  $\{x,y\}$ for which $|\overline{R}(x, y)|=m$, and  $\overline{R}(x, y)$ is not a resolving set of $G$. Then, ${\rm res}(G)\geq m+1$.
\end{proof}

\begin{corollary}\label{coro}
The resolving number ${\rm res}(G)$ can be computed in polynomial time in the order of $G$.
\end{corollary}

\begin{proof}
First, we preprocess the distance matrix of $G$ in $O(n^3)$ time (see \cite{floyd}).
Thus, for each pair $\{x,y\}\in \mathcal{P}_2(V(G))$ the set  $\overline{R}(x,y)$ can be obtained in $O(n)$ time by comparing the corresponding rows of $x$ and
$y$ in the distance matrix. By Proposition \ref{expr}, we can compute ${\rm res}(G)$ in $O(n^3)$ time since $|\mathcal{P}_2(V(G))|$ is  $O(n^2)$.
\end{proof}

We now provide three results which will be useful in the proofs of this paper.

\begin{lemma}\label{key}
Let $\mathcal{P}\subseteq \mathcal{P}_2(V(G))$ and let $V_1,V_2,\dots,V_{\ell}$ be a partition of
 $V(G)$. If every vertex of $V_i$, for $i=1, \ldots , \ell$, does not resolve at least $k_i\geq 0$ pairs of $\mathcal{P}$, then
 $$\sum_{i=1}^{\ell}|V_i|\cdot k_i\leq |\mathcal{P}|\cdot ({\rm res}(G)-1).$$
\end{lemma}

\begin{proof}
Every vertex $u\in V_i$ belongs to at least $k_i$ different sets of the form $\overline{R}(x,y)$ with $\{x,y\}\in \mathcal{P}$. Moreover, by Proposition \ref{expr}, we have $|\overline{R}(x,y)|\leq {\rm res}(G)-1$. Hence
$$\sum_{i=1}^{\ell}|V_i|\cdot k_i\leq \sum_{\{x,y\}\in \mathcal{P}} |\overline{R}(x,y)|\leq |\mathcal{P}|\cdot ({\rm res}(G)-1).$$
\end{proof}

Khuller et al.~\cite{landmarks} showed that $d(u,w)\in\{d-1,d,d+1\}$ for $u,v,w\in V(G)$ such that $\{v,w\}\in E(G)$ and $d(u,v)=d$. The following  straightforward lemma provides a version of this result for subsets of $V(G)$.

\begin{lemma}\label{general}
Let $u\in V(G)$ and $A\subseteq V(G)$. If $d(u,A)=d$ then $d(u,v)\in\{d,d+1,\ldots,d+d(A)\}$ for every $v\in A$.
\end{lemma}

A resolving set $S$ of $G$ is  \emph{minimal} if no proper subset
of $S$ is a resolving set. The maximum cardinality of a minimal resolving set is the \emph{upper dimension}
${\rm dim}^+(G)$ (see \cite{upper}). Clearly, every $(n - 1)$-subset of $V(G)$ is a resolving set  and
every resolving set contains a minimal resolving set. Hence,
$$1\leq {\rm dim}(G) \leq {\rm dim}^+(G) \leq {\rm res}(G) \leq n
- 1.$$
 This relation and several results in  \cite{upper} and \cite{Jannesarion} give all graphs $G$ with ${\rm res}(G)\leq 2$. Indeed, by Theorem 1 of \cite{Jannesarion}, the paths $P_1$ and $P_2$ are the only graphs with resolving number equal to 1. By Proposition 2.1 of \cite{upper} we have ${\rm res}(P_n)=2$ for $n\geq 3$ and ${\rm res}(C_n)=2$ if $n$ is odd. Theorem 2.4 of \cite{upper} says that ${\rm dim}^+(G) = {\rm res}(G)=2$ if and only if $G$ is a
path of order at least 4 or an odd cycle. If ${\rm dim}^+(G)=1$ and ${\rm res}(G)=2$ then ${\rm dim}(G)=1$ and, by Theorem A of \cite{upper}, one obtains that $G$ is a path.  Thus, we have proved the following result.

\begin{theorem}\label{res=2}
${\rm res}(G)\leq 2$ if and only if $G$ is a path or an odd cycle.
\end{theorem}

\section{Relationships between ${\rm res} (G)$ and other graph parameters}\label{sec:section3}

\subsection{Diameter}

In \cite{MDUDRN}, we showed that $d(G)\leq  3{\rm res}(G)-5$ for a graph $G$ (not being a cycle) with ${\rm res}(G)\geq 3$.
Here, we improve this bound for trees characterizing also the extremal case.

For positive integers $a,b,c$, let $S_{a,b,c}$ be a spider with three legs of lengths $a, b, c$, respectively (i.e., a tree formed by three paths of lengths $a,b,c$ attached at a single vertex).

\begin{proposition}\label{diametertree} If $T$ is a tree that is not a path then $d(T)\leq 2{\rm res}(T)-4$,
 and equality holds if and only if $T\cong S_{a,b,b}$ with $b={\rm res}(T)-2$ and $a\leq b$.
\end{proposition}

\begin{proof}
For simplicity, let us denote $r={\rm res} (T)$. By Theorem \ref{res=2} we have $r\geq 3$. Suppose on the contrary that $d(T)\geq 2r-3$. Assume, without loss of generality, that $d(T)=2r-3$ and let $P=(u_1,u_2,...,u_{2r-2})$ be a shortest path of length $d(T)$.  Clearly, $\delta(u_1)=\delta(u_{2r-2})=1$. Further, since $T$ is not a path then
there is a vertex $u\in N[u_i]\setminus P$ with $1<i<2r-2$. If $i\geq r$ (analogous for $i<r$) then no vertex of $S=\{u_1,u_2,...,u_i\}$ resolves the pair $\{u,u_{i+1}\}$ and so $S$ is not a resolving set of $G$; a contradiction with $|S|=i\geq r$. Therefore, $d(T)\leq 2r-4$.

One can easily check that $d(S_{a,r-2,r-2})= 2{\rm res}(S_{a,r-2,r-2})-4$ for $a\leq r-2$.
Consider now a tree $T$ such that $d(T)=2r-4$, and let $P=(u_1,...,u_{2r-3})$ be a shortest path of length $d(T)$. Next, we prove that $T\cong S_{a,r-2,r-2}$ with  $a\leq r-2$.

Arguing as above, we obtain $\delta(u_1)=1$, $\delta(u_{2r-3})=1$, and
there is a vertex $u\in N[u_i]\setminus P$ with $1<i<2r-3$. Moreover, the sets $\{u_1,...,u_r\}$ and $\{u_{r-2},...,u_{2r-3}\}$ are not resolving sets when $i\geq r$ and $i\leq r-2$, respectively. This implies $u_i=u_{r-1}$, and so $\delta(u_{r-1})\geq 3$ and $\delta(u_j)=2$ for $j\neq 1,r-1,2r-3$. That $\delta(u_{r-1})=3$ follows from the fact that $P$ is a resolving set, and the same argument shows that the induced subgraph $\langle V(T)\setminus P \rangle$ is a path. Thus, $T\cong S_{a,r-2,r-2}$ for some positive integer $a$.
Further, no vertex of the set $(V(T)\setminus P)\cup \{u_{r-1}\}$ resolves the pair $\{u_{r-2},u_{r}\}$ and so it is not a resolving set. Hence, $a=|V(T)\setminus P|\leq r-2$.
\end{proof}

\subsection{Girth}

We now provide a tight bound on the girth of a graph in terms of its
resolving number.

\begin{proposition}\label{girth}
If $G$ is neither a tree nor a cycle then $g(G)\leq  2{\rm res}(G)-1$, and this bound is tight.

\end{proposition}

\begin{proof}
Let $r={\rm res} (G)$. Suppose on the contrary that
$g(G)\geq  2r$.
Assume, without loss of generality, that $g(G)=2r$ and consider a
cycle of minimum length $C=(u_1,...,u_{2r})$. Since
$G$ is not a cycle then there is a vertex $u$ adjacent to some vertex of $C$, say $u_r$.
Further, $C$ has minimum length and so $(u_1, u_2, \ldots ,u_{r+1})$ is a
shortest $u_1$-$u_{r+1}$  path. Thus, the pair $\{u, u_{r+1}\}$ is resolved by no vertex of $S=\{u_1,...,u_r\}$, and so  $S$ is not a
 resolving set; a contradiction with $|S|=r$.

The graph obtained by attaching a pendant edge to any vertex of an odd cycle
$C_{2a+1}$ with $a\geq 3$ has girth $2a+1$ and resolving number
$a+1$. This proves that the bound is tight.
\end{proof}

\subsection{Clique number}

Jannesari and Omoomi~\cite{Jannesarion} proved that
 $\omega (G)\leq {\rm res} (G)+1$ for a graph $G$ such that ${\rm dim} (G)={\rm res} (G)$, and that equality holds only for complete graphs. However, their proof does not use the hypothesis ${\rm dim}(G)={\rm res}(G)$ but only the fact that every set $S\subseteq V(G)$ with $|S|\geq {\rm res} (G)$ is a resolving set of $G$. Thus, their result can be extended to every graph $G$.

\begin{proposition}\label{clique1}
For every graph $G$, $\omega (G)\leq {\rm res} (G) +1$ and equality holds if and only if $G$ is a complete graph.
\end{proposition}

 As a next step,
we determine all graphs $G$ such that $\omega (G)={\rm res} (G)$.
 Let $\mathcal{F}_1$ denote the set of 14 graphs depicted in Figure~\ref{F1}.
For  positive integers $a,b$ with $b<a$, let $G_{a,b}$ be the graph obtained
by attaching a vertex to any $b$ vertices of a complete graph on $a$ vertices (note that ${\rm res}(G_{a,b})=a$); see Figure~\ref{G1G2}(right) for a small example. Let $G^1$, $G^2$, $G^3$ and $G^4$ be the other four graphs illustrated in Figure~\ref{G1G2}.

\begin{figure}[ht]
\begin{center}
\includegraphics[width=120mm]{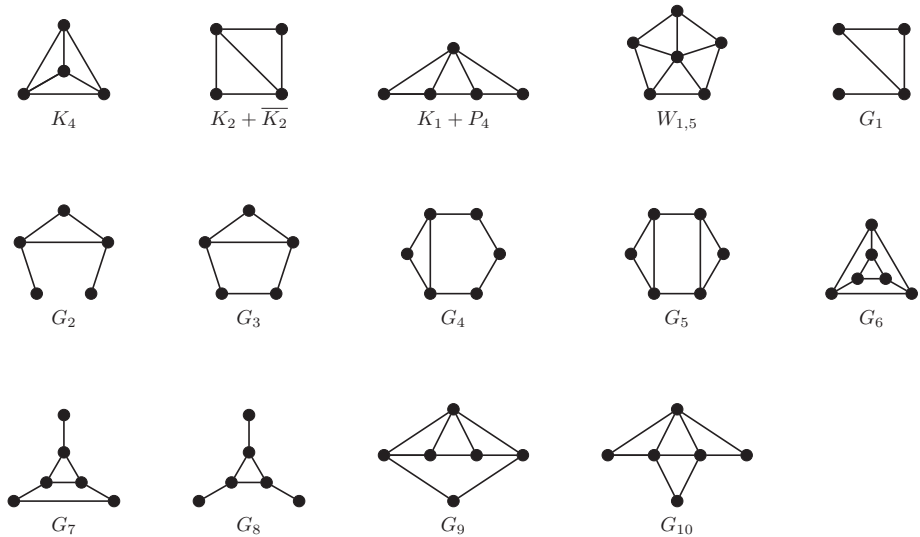}
\caption{The set of graphs $\mathcal{F}_1$.}\label{F1}
\end{center}
\end{figure}

\begin{figure}[ht]
\begin{center}
\includegraphics[width=120mm]{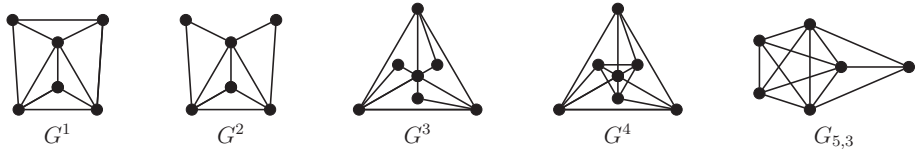}
\caption{The graphs $G^1$,  $G^2$, $G^3$, $G^4$ and $G_{5,3}$.}\label{G1G2}
\end{center}
\end{figure}

\begin{proposition}\label{clique2}
For every graph $G$, the following statements hold.
\begin{enumerate}
\item[(i)] $\omega (G)={\rm res} (G)=1$     $\Longleftrightarrow$ $G\cong K_1$.
\item[(ii)] $\omega (G)={\rm res} (G)=2$     $\Longleftrightarrow$ $G\cong P_n$ for some $n\geq 3$ or $G\cong C_n$ for some odd $n\geq 5$.
\item[ (iii)] $\omega (G)={\rm res} (G)=3$     $\Longleftrightarrow$ $G\in\mathcal{F}_1$.
\item[(iv)] $\omega (G)={\rm res} (G)=4$     $\Longleftrightarrow$ $G\in \{G^1,G^2,G^3,G^4\}$ or $G\cong G_{4,b}$ for some $b<4$.
\item[(v)] $\omega (G)={\rm res} (G)\geq 5$ $\Longleftrightarrow$ $G\cong G_{a,b}$ for some $b<a$ and $5\leq a={\rm res}(G)$.
\end{enumerate}
\end{proposition}

\begin{proof}
Statement (i) is straightforward, and Statement (ii) follows immediately from Theorem \ref{res=2} (note that ${\rm res}(P_2)=1$). As a consequence of the study developed in Section \ref{sec:section4} (see Theorem \ref{res=3}) one can easily prove Statement (iii).
Further, if $G$ is isomorphic to either $G^1$, $G^2$ or $G_{a,b}$ with ${\rm res} (G)=a\geq 4$ and $b<a$, then the corresponding statements can also be easily checked.

Consider now a graph $G$ such that $\omega (G)={\rm res} (G)=r\geq 4$. Let $K$ be a maximum clique in $G$. By Proposition \ref{clique1}, the graph $G$ is not a complete graph and so there is a vertex $u\in V(G)\setminus K$ adjacent to some vertex of $K$. If $n=r+1$ then $G\cong G_{r,b}$ for some  $b<r$.
Otherwise, there is a vertex $v\in V(G)\setminus (K\cup\{u\})$
 such that either $v\in N(u)$ or $v$ is adjacent to some vertex of $K$.

Let $A_1=(N(u)\setminus N(v))\cap K$, $A_2=(N(v)\setminus N(u))\cap K$, $A_3=N(u)\cap N(v)\cap K$ and $A_4=K\setminus (N(u)\cup N(v))$. Note that any set $A_i$ may be empty, and  $K$ is the disjoint union of the four sets. We claim that $|A_i|\leq 1$ for every $1\leq i\leq 4$.
Indeed, suppose on the contrary that there are two different vertices $x, y$ in, say $A_1$.
Clearly, $d(v,x)=d(v,y)=2$ and so the set $S=(K\setminus \{x,y\})\cup \{u,v\}$ is not a resolving set; a contradiction with $|S|=r$.
Hence, $|A_1|\leq 1$ and a similar argument applies to the remaining sets $A_i$.

Since $\sum_{i=1}^4 |A_i|=|K|=r\geq 4$ and $|A_i|\leq 1$ then $|K|=r=4$ and $|A_i|=1$ for every $1\leq i \leq 4$. Hence, we can assume that $N(u)\cap K=\{u_1,u_2\}$ and $N(v)\cap K=\{u_2,u_3\}$ for $K=\{u_1,u_2,u_3,u_4\}$.

If $n=6$ then
$G$ is isomorphic to either $G^1$ or $G^2$ (depending on whether $u$ and $v$ are adjacent or not). Suppose now that $n\geq 7$ and let $w\in V(G)\setminus (K\cup \{u,v\})$. Considering the analogous sets $A_i$ but for the vertices $u,w$ and $v,w$, we deduce that either $N(w)\cap K=\{u_1,u_3\}$ or $N(w)\cap K=\{u_2,u_4\}$. This implies $n=7$ (otherwise there is a vertex $z\in V(G)\setminus (K\cup \{u,v,w\})$ such that $A_3=N(w)\cap N(z)\cap K=\emptyset$ and $A_4=K\setminus (N(w)\cup N(z))=\emptyset$).

We next show that it cannot be the case that $N(w)\cap K=\{u_1,u_3\}$ by providing a set $S$ with $|S|=4$ that is not a resolving set. Consider the induced subgraph $H=\langle\{u,v,w\}\rangle$. If $|E(H)|\in \{0,1\}$, say $E(H)=\{\{u,v\}\}$ for $|E(H)|=1$, then no vertex of $S=\{u_1,u_3,u,v\}$ resolves the pair $\{u_4,w\}$. If $|E(H)|\in\{2,3\}$ (assume  $E(H)=\{\{u,v\},\{v,w\}\}$ for $|E(H)|=2$)
then no vertex of $S=\{u_2,u_3,u,w\}$ resolves the pair $\{u_1,v\}$ .

When $N(w)\cap K=\{u_2,u_4\}$, we obtain the graphs $G^3$ and $G^4$  for $|E(H)|=0$ and
$|E(H)|=3$, respectively. If $|E(H)|\in \{1,2\}$, say $E(H)=\{\{u,v\}\}$ or $E(H)=\{\{u,v\},\{v,w\}\}$, then either
 $\{u_2,u_3,u,w\}$ or $\{u_1,u_2,v,w\}$ is not a
resolving set; a contradiction with $r=4$.\end{proof}

\subsection{Order}

The following proposition gives bounds on the order of a graph in terms of its resolving number, and also provides an alternative proof for Theorem~3.7 of \cite{MDUDRN} which states that
the set of graphs with resolving number $a\geq 4$ is finite.

\begin{proposition}\label{n}
If $G$ is neither a path nor a cycle then
$${\rm res}(G)+1\leq n \leq \left\{ \begin{tabular}{lcl}
 $3{\rm res} (G)-3$ &   if  & $g (G)=3$, \\
 $ 4{\rm res} (G)-4$ &   if  & $g (G)=4$,\\
 $5{\rm res} (G)-5 $&   if  & $g (G)=5$, \\
  $5{\rm res} (G)-9$ &   if  & $g (G)>5$ \, and \, $\Delta (G) >3,$\\
  $6{\rm res}(G) -8$ &  if & $g (G)> 5$ \,  and \, $\Delta (G)=3.$\\
\end{tabular}
\right.$$   Moreover, the lower bound and the upper bounds for $g (G)\in\{3,5\}$ are tight.
\end{proposition}

\begin{proof}
By definition ${\rm res}(G)+1\leq n$ for every graph $G$, and the complete graph $K_n$ attains the bound (see \cite{upper}). For the upper bound, we distinguish several cases.

\emph{Case 1.} $g(G)=3$: Let $C=(u_1,u_2,u_3)$ be a 3-cycle in $G$, and $u\in V(G)$. Since $d(C)=1$, by Lemma~\ref{general}, $d (u,u_i)\in \{d(u,C), d(u,C)+1\}$ for every  $1\leq i\leq 3$. Hence, there is at least one pair $\{u_i,u_j\}\in \mathcal{P}_2(C)$ which is not resolved by vertex $u$.
Lemma~\ref{key} gives the expected bound by setting $\mathcal{P}=\mathcal{P}_2(C)$, $V_1=V(G)$ ($\ell=1$) and $k_1=1$.

The graph consisting of a 3-cycle with one path of length $r-2$  attached to each vertex of the cycle has resolving number $r$ and $3(r-1)$ vertices. This shows that the bound is tight.

\emph{Case 2.} $g(G)=4$: Let $C=(u_1,u_2,u_3,u_4)$ be a cycle of minimum length in $G$, $V_1=R(u_1,u_3)\cap R(u_2,u_4)$
 and $V_2=V(G)\setminus V_1$. We now distinguish two cases.

\,

\emph{Case 2.1.} $|V_1|\leq |V_2|$: Every vertex $u\in V_2$ does not resolve at least one pair of $\mathcal{P}=\{\{u_1,u_3\},\{u_2,u_4\}\}$. By Lemma~\ref{key}, taking $k_1=0$ and $k_2=1$ ($\ell=2$) we have
$2({\rm res}(G)-1)\geq |V_2|\geq n/2$.

\,

\emph{Case 2.2.}: $|V_1|> |V_2|$: Let $u\in V_1$ and $d(u,C)=d$. Assume that $d(u,u_1)=d$. Then $d(u,u_2), d(u,u_4)\in \{d,d+1\}$. Since
$u\in R(u_2,u_4)$ then either $d(u,u_2)=d$ or $d(u,u_4)=d$. Moreover, $u\in R(u_1,u_3)$ and so $d(u,u_3)=d+1$. Therefore, there are at least two pairs of $\mathcal{P}=\{\{u_1,u_2\},\{u_2,u_3\},\{u_3,u_4\},\{u_4,u_1\}\}$ which are not resolved by vertex $u$. By Lemma~\ref{key}, setting $k_1=2$ and $k_2=0$ ($\ell=2$) we have $4({\rm res}(G)-1)\geq 2|V_1|>n$.

 \,

\emph{Case 3.} $g(G)=5$: A 5-cycle $C=(u_1,u_2,u_3,u_4,u_5)$ has diameter 2 and so  Lemma~\ref{general} gives $d (u,u_i)\in \{d(u,C), d(u,C)+1, d(u,C)+2\}$ for every  $u\in V(G)$ and $1\leq i\leq 5$. Hence, there are at least two pairs of $\mathcal{P}_2(C)$ which are not resolved by vertex $u$. The result follows by Lemma~\ref{key}, taking $\mathcal{P}=\mathcal{P}_2(C)$, $V_1=V(G)$ ($\ell=1$) and $k_1=2$.
 The graph $G_{14}$ depicted in Figure~\ref{F2} (see Section \ref{sec:section4}) shows that the bound is tight.

\,

\emph{Case 4.} $g(G)>5$ and $\Delta (G) >3$: Let $u_0\in V(G)$ with $\delta(u_0)\geq 4$, and consider four of its neighbours, say
$u_1, u_2, u_3, u_4$. Let $A=\{u_0,u_1,u_2,u_3,u_4\}$. Arguing as in Case 3,
since $d(A)=2$ then every vertex $u\in V(G)$ does not resolve at least two pairs of $\mathcal{P}_2(A)$.
Moreover, $u_0$ resolves no pair of $\mathcal{P}_2(A\setminus \{u_0\})$, and $u_i$  resolves
no pair of $\mathcal{P}_2(A\setminus\{u_0,u_i\})$ for every $1\leq i \leq 4$. By Lemma~\ref{key}, setting $\mathcal{P}=\mathcal{P}_2(A)$, $V_1=V(G)\setminus A$, $V_2=\{u_0\}$, $V_3=A\setminus \{u_0\}$ ($\ell=3$), $k_1=2$, $k_2=6$ and $k_3=3$, we have $10({\rm res}(G)-1)\geq 2(n-5)+6+12=2n+8.$

\,

\emph{Case 5.} $g(G)>5$ and $\Delta (G)=3$: Let $u_0\in V(G)$ with   $\delta(u_0)=3$ and neighbours $u_1, u_2, u_3$. Reasoning as above, every vertex $u\in V(G)$ does not resolve at least one pair of $\mathcal{P}_2(N[u_0])$. Further, $u_0$ resolves no pair of $\mathcal{P}_2(N(u_0))$. Lemma~\ref{key} gives the expected bound by taking $\mathcal{P}=\mathcal{P}_2(N[u_0])$, $V_1=V(G)\setminus \{u_0\}$, $V_2=\{u_0\}$ ($\ell=2$), $k_1=1$ and $k_2=3$.
\end{proof}

We now provide a tight bound for trees, characterizing those that attain the bound.

\begin{proposition}\label{ordertree} If $T$ is a tree of order $n$ that is not a path then
$n\leq 3 {\rm res} (T)-5$, and equality holds if and only if $T\cong S_{a,a,a}$ with  $a={\rm res}(T)-2\geq 1$.
\end{proposition}

\begin{proof}
Let $u\in V(T)$ such that $\delta(u)\geq 3$ and consider three of its neighbours, say $u_1,u_2,u_3$. We can partition $V(T)$
into four subsets $U, U_i$ with  $1\leq i\leq 3$, where $U_i$ contains all vertices $v\in V(T)$ such that $u_i$ belongs to the $u$-$v$ path, and $U=V(T)\setminus \cup_{i=1}^3 U_i$. Note that $u_i\in U_i$ and $u\in U$.
Clearly, no vertex of $U_i$ resolves the pair $\{u_{j},u_{k}\}$ for $i\neq j\neq k$. Moreover, every vertex of $U$ resolves no pair of $\mathcal{P}_2(\{u_1,u_2,u_3\})$. Since every subset of vertices with cardinality at least ${\rm res}(T)$ is a resolving set of $T$ we can conclude that for $1\leq i\leq 3$,
\begin{equation}
|U_i|+|U|\leq {\rm res}(T)-1.
\end{equation}
Hence, $\sum_{i=1}^3|U_i|+3|U|\leq 3({\rm res}(T)-1)$ and so
\begin{equation}
n\leq 3({\rm res}(T)-1)-2|U|\leq 3({\rm res}(T)-1)-2.
\end{equation}

If $T$ is isomorphic to $S_{a,a,a}$ with  $a={\rm res}(T)-2$ then $n=3{\rm res}(T)-5$.  Consider now a tree $T$ such that $n= 3 {\rm res} (T)-5$. Proceeding as above, one can partition $V(T)$  into the subsets $U,U_i$ with $1\leq i\leq 3$. Since $T$ satisfies expression (2) then $n$ attains its maximum value when $|U|= 1$, i.e., $U=\{u\}$.  Therefore, $$3{\rm res}(T)-5=n=\sum_{i=1}^3|U_i|+1.$$ By expression (1) we have
 $|U_i|+1\leq {\rm res}(T)-1$ which leads to $|U_i|={\rm res}(T)-2$ for $1\leq i\leq 3$.

Clearly, the induced subgraph $\langle U_i \rangle$  is a path in $T$ for every $1\leq i\leq 3$. Indeed, if it were the case that two vertices $x,y\in U_i$ verify that $d(u,x)=d(u,y)$, then the set $U_{j}\cup U_{k}\cup \{u\}$ for $i\neq j\neq k$ would not be a resolving set; a contradiction with $|U_{j}|+|U_{k}|+1=2{\rm res}(T)-3\geq {\rm res}(T)$ since, by Theorem \ref{res=2}, ${\rm res}(T)\geq 3$. Hence,
$T$ is a tree which consists of three paths of length ${\rm res}(T)-2$ attached at a single vertex, i.e., $T\cong S_{a,a,a}$  with $a ={\rm res}(T)-2\geq 1$.
\end{proof}

\subsection{Maximum degree}

Jannesari and Omoomi~\cite{Jannesarichar}  proved that  $\Delta (G)\leq 2^{r-1}+r+1$ for $r={\rm res}(G)$. Here, we improve this exponential bound providing a linear bound.

\begin{proposition}\label{Delta}
If $G$ is neither a path nor a cycle then
$$\Delta (G) \leq \left\{ \begin{array}{lcc}
3{\rm res}(G)-4 &   if  & g (G)= 3, \\
{\rm res}(G)  &  if & g (G)>3,
\end{array}
\right.$$
and both bounds are tight.\end{proposition}

\begin{proof}
Proposition \ref{n} gives the bound for $g(G)=3$
since $\Delta (G)\leq n-1$. The wheel graph $W_{1,5}$ (see Figure \ref{F1}) attains the bound.

Suppose now that $g (G) >3$. Let $u_0$ be a vertex of degree $\Delta (G)$, and $u_1, u_2\in N(u_0)$.
Clearly, the induced subgraph $\langle N[u_0] \rangle$ is a star on $\Delta (G)+1$ vertices (otherwise $G$ would contain a triangle). Then, no vertex of $S=N[u_0]\setminus\{u_1,u_2\}$ resolves the pair $\{u_1,u_2\}$ and so $S$ is not a resolving set of $G$. Therefore,
$|S|=\Delta (G)-1\leq {\rm res} (G)-1$. The star $K_{1,a}$ on $a+1$ vertices proves that the bound is tight since ${\rm res}(K_{1,a})=\Delta(K_{1,a})=a$.
\end{proof}

We now show that only stars attain the bound when restricting the preceding result to trees.

\begin{proposition}\label{deltatree} If $T$ is a tree of order $n$ that is not a path then $\Delta(T)\leq {\rm res}(T)$,
 and equality holds if and only if $T$ is isomorphic to a star $K_{1,a}$ with $a={\rm res}(T)\geq 3$.
\end{proposition}

\begin{proof}
The fact that $\Delta(T)\leq {\rm res}(T)$ follows immediately from Proposition~\ref{Delta} since $g(T)$ is defined to be infinity. Also, as mentioned above ${\rm res}(K_{1,a})=\Delta(K_{1,a})=a$. Thus, it suffices to prove that a tree $T$ (not being a path) with  $\Delta(T)={\rm res}(T)$
is isomorphic to  $K_{1,a}$ with $a= {\rm res}(T)\geq 3$.

Let $u_0$ be a vertex of degree $\Delta (T)$, and  $N[u_0]=\{u_0, u_1,...,u_{\Delta (T)}\}$ its closed neighbourhood. Suppose that there is a vertex $u\in V(T)$ such that $d(u,u_0)=2$. Assume, without loss of generality, that $u\in N(u_1)$. Then, the  set $S=(N[u_0]\setminus\{u_2,u_3\})\cup\{u\}$ is not a resolving set since the pair $\{u_2,u_3\}$ is resolved by no vertex of $S$.
This contradicts  $|S|=\Delta (T)={\rm res} (T)$ and so there does not exist such a vertex $u$, which implies that $T$ is a star $K_{1,a}$ with
$a=\Delta (T)={\rm res}(T)\geq 3$.
 \end{proof}

\section{Characterization of the graphs $G$ with ${\rm res}(G)=3$}\label{sec:section4}

As a natural extension of Theorem \ref{res=2}, we now determine all graphs with resolving number 3 using as main tools the relations obtained in Section \ref{sec:section3}. We begin with three technical lemmas.

\begin{lemma}\label{girth3}
If $g (G)=\Delta (G)={\rm res} (G)=3$ then $G\in \mathcal{F}_1$.
\end{lemma}

\begin{proof}
Let $C=(u_1,u_2,u_3)$ be a cycle in $G$. Suppose first that  $\delta (u_1)=3$ and $\delta(u_2)=\delta(u_3)=2$, and let $u\in N(u_1)\setminus \{u_2,u_3\}$. Clearly $n=4$, since otherwise no vertex $v\in N(u)\setminus \{u_1\}$ resolves the pair $\{u_2,u_3\}$ and so the set  $\{u_1,u,v\}$ is not a resolving set, contradicting ${\rm res}(G)=3$.
Therefore, $G\cong G_1\in \mathcal{F}_1$. Recall that the set of graphs
$\mathcal{F}_1$ is illustrated in Figure \ref{F1}.

Assume now that $\delta (u_1)=\delta(u_2)=3$ and $\delta(u_3)=2$. We distinguish two cases.

\emph{Case 1.} $N(u_1)\cap N(u_2)=\{u_3\}$: Let $N(u_1)=\{u_2,u_3,u\}$ and $N(u_2)=\{u_1,u_3,v\}$. Since $g(G)=3$, by Proposition ~\ref{n}, we have $n\leq 6$. If $n=5$ then $G$ is isomorphic to either $G_2$ or $G_3$ (both in $ \mathcal{F}_1$). If $n=6$ then there is a vertex $w\in (N(u)\cup N(v))\setminus \{u_1,u_2\}$. Further, $w\in N(u)\cap N(v)$ (otherwise either $\{u_1,u,w\}$ or $\{u_2,v,w\}$ is not a resolving set). Hence,  $G$ is isomorphic to either $G_4$ or $G_5$ (both contained in $\mathcal{F}_1$).

\,

\emph{Case 2.}  $N(u_1)\cap N(u_2)=\{u_3,u\}$: If it were the case that $n\geq 5$ then there would be a vertex $v\in V(G)$ such that $N(u)=\{u_1,u_2,v\}$, and the set $\{u_3,u,v\}$ would not be a resolving set. Therefore, $n=4$ and so $G\cong K_2+\overline{K}_2\in \mathcal{F}_1$.

\,

Suppose, finally, that $\delta(u_1)=\delta(u_2)=\delta(u_3)=3$. By Proposition~\ref{clique1} it follows that $G\cong K_4 \in \mathcal{F}_1$
when $N(u_1)\cap N(u_2) \cap N(u_3)=\{u\}$. Moreover, arguing as in Case 2 (when $n\geq 5$), we deduce that it cannot be the case that
 two of the $u_i$'s, say $u_1, u_2$,  have a common neighbour $u\notin N(u_3)$.
 Assume then that $N(u_i)\cap N(u_j)=\{u_k\}$ for all $i\neq j\neq k$.

 Let $N(u_1)=\{u_2,u_3,u\}$, $N(u_2)=\{u_1,u_3,v\}$ and $N(u_3)=\{u_1,u_2,w\}$. Clearly, the size of the induced subgraph $H=\langle \{u,v,w\} \rangle$ is either 3, 1 or 0.  Indeed, if $|E(H)|=2$, say $E(H)=\{\{u,v\},\{u,w\}\}$, then $S=\{u_2,u,w\}$ is not a resolving set since no vertex of $S$ resolves the pair $\{u_1,v\}$. Hence, $G$ is isomorphic to either $G_6$, $G_7$ or $G_8$ (all contained in $\mathcal{F}_1$).
\end{proof}

\begin{lemma}\label{delta4}
If $\Delta (G)=4$ and ${\rm res} (G)=3$ then $G\in \mathcal{F}_1$.
\end{lemma}

\begin{proof}
Let $u_0\in V(G)$ be a vertex with $\delta(u_0)=4$. By Lemma 4 of \cite{Jannesarichar}, the induced subgraph $\langle N(u_0) \rangle$ is a path, say $(u_1,u_2,u_3,u_4)$. Then $g(G)=3$ and, by Proposition~\ref{n},  either $n=5$ or $n=6$.
If $n=5$ then $G\cong K_1+P_4\in \mathcal{F}_1$. If $n=6$ then there is a vertex $u$ such that $d(u,u_0)=2$. Only when $\delta (u)=2$ and $N(u)$ is either $\{u_1,u_4\}$ or $\{u_2,u_3\}$ we obtain two possible graphs $G$ which are contained in $\mathcal{F}_1$: the graphs
$ G_{9}$ and $G_{10}$ (see Figure \ref{F1}).
In the remaining cases, we next specify a set $S$ with $|S|\geq 3$ that is not a resolving set, obtaining a contradiction.

If $\delta (u)=1$, take $S=\{u_0,u_1,u\}$ for $N(u)=\{u_1\}$, and $S=\{u_0,u_2,u\}$ for $N(u)=\{u_2\}$ (analogous for  $N(u)=\{u_4\}$ and $N(u)=\{u_3\}$). If $\delta (u)=2$, consider $S=\{u_2,u_4,u\}$ for $N(u)=\{u_1,u_2\}$, and $S=\{u_0,u_2,u\}$ for  $N(u)=\{u_1,u_3\}$ (similar for $N(u)=\{u_3,u_4\}$ and $N(u)=\{u_2,u_4\}$). Finally, if $\delta (u)\geq 3$ take  $S=N(u)$.
 \end{proof}

Consider now the set of graphs $\mathcal{F}_2$ shown in Figure~\ref{F2}.

\begin{lemma}\label{girth5}
If $g (G)=5$ and $\Delta (G)={\rm res} (G)=3$ then  $G\in \mathcal{F}_2$.
\end{lemma}

\begin{proof}
We first observe that any two different 5-cycles in $G$ meet in at most one edge. Indeed,
suppose on the contrary that there are two 5-cycles meeting in two edges $e, e'\in E(G)$. Since $g(G)=5$, one can easily check that the edges $e, e'$ must be consecutive, and then the set formed by the three end vertices of $e$, $e'$ is not a resolving set, which contradicts ${\rm res} (G)=3$.

Consider now a vertex $u_0\in V(G)$ with $\delta(u_0)=3$ and such that
the number of 5-cycles through $u_0$ is maximum. Let $N(u_0)=\{u_1,u_2,u_3\}$.
 We distinguish three cases.

\emph{ Case 1.} Suppose that there exists a unique 5-cycle through $u_0$, say $C=(u_0,u_1,u,v,u_2)$. If $n=6$ then $G \cong G_{11}\in\mathcal{F}_2$. Otherwise, let $w\in V(G)\setminus (N[u_0]\cup \{u,v\})$.

Clearly, $w\notin N(u_i)$ for $1\leq i\leq 3$:
 the sets $\{u_0,u_1,w\}$,  $\{u_0,u_2,w\}$, $\{u_0,u_3,w\}$ would not be resolving sets if it were the case that $w$ is adjacent to, respectively, $u_1$, $u_2$, $u_3$ (recall that $u_0$ belongs to exactly one 5-cycle).

Assume now that $w\in N(u)$ (analogous for $w\in N(v)$). If $n=7$ then $G\cong G_{12}\in\mathcal{F}_2$. When $n\geq 8$, there is a vertex $x\in V(G)\setminus (N[u_0]\cup \{u,v,w\})$. Arguing as above, we conclude that $x\notin N(u_i)$ for every $1\leq i\leq 3$, and so either $x\in N(w)$ or
$x\in N(v)$. In both cases, one reaches a contradiction with ${\rm res} (G)=3$. Indeed, if $x\in N(w)$ then $d(x,u_1)=d(x,v)=3$ since there is no 5-cycle going through $u$ other than $C$. This implies that $\{u,w,x\}$ is not a resolving set. A similar reasoning gives that $\{u,v,x\}$ is not a resolving set when $x\in N(v)$.

\,

\emph{ Case 2.} Suppose that there are exactly two different 5-cycles through $u_0$, say $C_1=(u_0,u_1,u,v,u_2)$ and $C_2=(u_0,u_2,u',v',u_3)$. If $n=8$ then $G\cong G_{13}\in\mathcal{F}_2$. Otherwise,
there exists a vertex $w\in V(G)\setminus (C_1\cup C_2)$. Distinguishing cases according to the adjacencies of $w$, we next
give a set $S$ with $|S|=3$ that is not a resolving set; a contradiction with ${\rm res} (G)=3$.
 One only has to note that: (1) $u_0$ belongs to exactly two $5$-cycles, (2) two 5-cycles share at most one edge, (3) the number of 5-cycles through $u_0$ is maximum.

If $w\in N(u_1)$ (analogous for $w\in N(u_3)$) then no vertex of $S=\{u_0,u_1,u_3\}$ resolves the pair $\{u,w\}$. If $w\in N(u)$ (similar for $w\in N(v')$) then the pair $\{u_0,u_2\}$ is resolved by no vertex of $S=\{u, v',w\}$. Finally, if $w\in N(v)$ (analogous for $w\in N(u')$) then no vertex of $S=\{u_2,v,u'\}$ resolves $\{u,w\}$.

\,

\emph{ Case 3.} There are three different 5-cycles $C_i$, $1\leq i\leq 3$, through $u_0$: By Proposition~\ref{n}, $n=10$. Moreover, arguing as above one can easily prove that the set $E(G)$ contains no edges other than those of the three cycles $C_i$. Therefore, $G\cong G_{14}\in \mathcal{F}_2$.
\end{proof}

\begin{figure}[ht]
\begin{center}
\includegraphics[width=115mm]{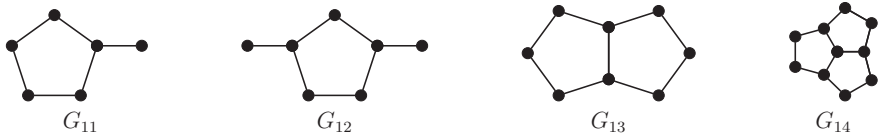}
\caption{The set of graphs $\mathcal{F}_2$.}\label{F2}
\end{center}
\end{figure}

We are now ready for proving the main result in this section.

\begin{theorem}\label{res=3}
${\rm res}(G)=3$ if and only if $G$ is an even cycle, the star $ K_{1,3}$ or $G\in\mathcal{F}_1\cup \mathcal{F}_2$.
\end{theorem}

\begin{proof}
It is easy to check that even cycles, $K_{1,3}$ and all graphs of $\mathcal{F}_1\cup \mathcal{F}_2$ have resolving number 3.

Consider now a graph $G$ with ${\rm res}(G)=3$. By Theorem \ref{res=2}, $G$ is neither a path nor an odd cycle. We can also assume that $G$ is not an even cycle (otherwise the result follows).

Proposition~\ref{Delta} yields $\Delta(G)\leq 5$. Moreover, $\Delta (G)> 2$ since the only connected graphs with maximum degree  2 are paths and cycles.

Suppose first that $\Delta (G)=3$. If $G$ is a tree, by Proposition~\ref{deltatree}, we have $G\cong K_{1,3}$. Otherwise, Proposition~\ref{girth} gives $g(G)\leq 5$. Lemmas~\ref{girth3} and \ref{girth5} lead to $G\in\mathcal{F}_1$ and  $G\in\mathcal{F}_2$ when $g(G)=3$ and $g(G)=5$, respectively.
Moreover, $g (G)\neq  4$. Indeed, suppose on the contrary that $g(G)=4$ and let $C=(u_1,u_2,u_3,u_4)$ be a minimum cycle in $G$. Assume that $\delta (u_1)=3$ and so there is a vertex $u\in N[u_1]\setminus C$. Hence, no vertex of the set $\{u_1,u_3,u\}$ resolves the pair $\{u_2,u_4\}$, which contradicts ${\rm res}(G)=3$.

By Lemma ~\ref{delta4}, we obtain $G\in\mathcal{F}_1$ when  $\Delta (G) =4$. Assume, finally, that $\Delta (G)=5$, and let $u\in V(G)$ with $\delta(u)=5$. Since ${\rm res}(G)=3$, by Lemma 3 of \cite{Jannesarichar}, we have that the induced subgraph $\langle N(u) \rangle$ is a 5-cycle. Then $g(G)=3$ and, by Proposition \ref{n}, it follows that $n=6$. Therefore, $G$ is isomorphic to the wheel graph $W_{1,5}\in \mathcal{F}_1$.
\end{proof}

\section{Concluding remarks}\label{sec:section5}

In this paper, we have studied a graph parameter related to the metric dimension: the resolving number.
We first establish an important difference between both parameters: ${\rm res} (G)$ is polynomial-time computable while computing ${\rm dim} (G)$ is NP-hard (see \cite{landmarks}). We then relate the resolving number to classical graph parameters, and characterize the graphs with resolving number 3 by using those relations. As it was said before, our study follows the same vein as several papers on metric dimension, and
most of our results either improve relationships obtained in other papers or continue with the studies developed in them.

Although  we provide an $O(n^3)$ time algorithm for computing the resolving number of an arbitrary graph,
it would be interesting to find the exact computational complexity of this parameter, even in specific families of graphs. Also, the non tight upper bounds of Proposition \ref{n} could be improved. Moreover, we use very small examples (the graph $G_{14}$ in Figure~\ref{F2} and the wheel graph $W_{1,5}$ in Figure \ref{F1}) to show the tightness of two bounds given in Propositions \ref{n} and \ref{Delta}, respectively, and so
it appears that they are not tight for large enough values of ${\rm res}(G)$.
 Finally, the problem of characterizing the graphs with fixed resolving number $a\geq 4$ remains open.

%\nocite{*}
\bibliographystyle{plain}
\bibliography{biblio-RNG}
\label{sec:biblio}

\end{document}